\documentclass[11pt]{amsart}
\usepackage{amsmath}
\usepackage{amssymb}
\usepackage{amsthm}
\usepackage{array}
\usepackage{xy}
\usepackage[pdftex]{graphicx}
\usepackage{hyperref}
\usepackage{color}
\usepackage{transparent}
\usepackage{latexsym}

\theoremstyle{plain}
  \newtheorem{theorem}{Theorem}[section]
  \newtheorem{lemma}{Lemma}[section]

  \newtheorem{corollary}{Corollary}[section]
  \newtheorem{definition}{Definition}[section]
  \newtheorem{remark}{Remark}[section]

   \newcommand{\beqn}{\begin{eqnarray}}
   \newcommand{\eeqn}{\end{eqnarray}}
   \newcommand{\beqs}{\begin{eqnarray*}}
   \newcommand{\eeqs}{\end{eqnarray*}}
   \newcommand{\ban}{\begin{eqnarray*}}
   \newcommand{\nan}{\end{eqnarray*}}
   \newcommand{\beq}{\begin{equation}}
   \newcommand{\eeq}{\end{equation}}
   
   \newcommand{\p}{\partial}

\newcommand{\Lm}{{\Omega^*}}
\newcommand{\Om}{\Omega}
\newcommand{\pom}{{\p\Om}}
\newcommand{\bom}{{\overline\Om}}

\renewcommand{\det}{\mbox{det}}
  
  \newcommand{\dist}{\mbox{dist}}
  \newcommand{\diam}{{\mbox{diam}}}

  \newcommand{\Vol}{{\mbox{Vol}}}

\newcommand{\R}{\mathbb{R}}

\numberwithin{equation}{section}

  \numberwithin{equation}{section}
  \numberwithin{figure}{section}

\begin{document}

\title[Regularity of free boundaries in optimal transportation]{\textbf{Regularity of free boundaries in optimal transportation}}

\author[S. Chen]
{Shibing Chen}
\address{School of Mathematical Sciences,
University of Science and Technology of China,
Hefei, 230026, P.R. China}
\email{chenshibing1982@hotmail.com}

\author[J. Liu]
{Jiakun Liu}
\address
	{
	School of Mathematics and Applied Statistics,
	University of Wollongong,
	Wollongong, NSW 2522, Australia}
\email{jiakunl@uow.edu.au}

\thanks{This work was supported by the Australian Research Council DP170100929 and DP200101084}



\thanks{}


\date{\today}

\dedicatory{}

 \begin{abstract}

In this paper, we obtain some regularities of the free boundary in optimal transportation with the quadratic cost. 
When $f, g \in C^\alpha$, and $\partial\Omega, \partial\Omega^*\in C^{1,1}$ are convex and far apart, by adopting our recent results on boundary regularity of Monge-Amp\`ere equations \cite{CLW1},we shows that the free boundaries are $C^{2,\alpha}$. 
Moreover, we obtain analogous regularities of the free boundary in an optimal transport problem with two separate targets. 
\end{abstract}

\maketitle

\baselineskip=16.4pt
\parskip=3pt

\section{introduction}
Let $\Omega, {\Omega^*} \subset\R^n$ be two bounded and convex domains.
Given a pair of $L^1$ density functions $f, g$ supported on $\Om, \Lm$ respectively, the optimal partial transport problem asks for the most economical way to transport a fraction 
\beq\label{mass}
m\leq \min\left\{\int_\Omega f(x)\,dx, \ \ \int_{\Omega^*} g(y)\,dy\right\}
\eeq 
of the mass of $f$ to $g$. 
Here the cost per unit mass transported is given by the quadratic distance squared function $c(x,y)=|x-y|^2/2$. 
More precisely, let $\Gamma_\leq(f,g)$ denote the set of nonnegative Borel measures $\gamma$ on $\R^n\times\R^n$ satisfying
 	$$\gamma(A\times \mathbb{R}^n) \leq \int_{A} f(x)\,dx, \quad \gamma(\mathbb{R}^n \times A) \leq \int_{A} g(y)\,dy$$ 
for any Borel set $A\subset\R^n$. 
The optimal partial transport problem is to minimise the cost functional
\begin{equation} \label{mini}
\mathcal{C}(\gamma) := \int_{\mathbb{R}^n \times \mathbb{R}^n} c(x,y)\, d\gamma(x,y)
\end{equation}       
among all $\gamma\in\Gamma_\leq(f,g)$ of fixed mass $\gamma(\R^n\times\R^n)=m$. 

In a pioneering work \cite{CM}, Caffarelli and McCann obtained the existence and uniqueness of the minimiser of \eqref{mini}. 
In particular, the optimal measure $\gamma$ of \eqref{mini} can be characterised by the minimiser $(u,v)$ of the following problem
\beq\label{dual mini}
\inf_{\substack{\psi(x)+\phi(y)\geq\langle x,y \rangle \\ \psi(x)\geq h(x), \ \ \phi(y)\geq h(y)}} \int_{\R^n}\psi(x)f(x)\,dx + \int_{\R^n}\phi(y)g(y)\,dy,
\eeq
where $h(x)=(|x|^2-\lambda)/2$ and $\lambda$ is a constant depending on $m$. 
Since $u$ can always be replaced by $\tilde u=\max\{v^*, h\}$ without increasing \eqref{dual mini}, where $v^*$ is the Legendre transform
\beq\label{lgdre}
v^*(x) := \sup_{y\in\R^n} \left\{x\cdot y - v(y)\right\},
\eeq
one may assume $u, v$ are convex functions. 
Then the optimal measure $\gamma$ of \eqref{mini} is supported on the graphs of $\nabla u$ over $U:=\{x\in\Om : u(x) > h(x)\}$, namely
\beqs
\gamma\left(\left\{(x, \nabla u(x)) : x\in U\right\}\right)=m=\int_U f(x)\,dx. 
\eeqs
The domain $U$ contains all mass $m$ been transported, and the hypersurface $\p U\cap\Omega$ dividing $\Om$ into two parts is called the \emph{free boundary}. 
Similarly, by duality one has $V := \{y\in\Lm : v(y)>h(y)\}$, and the free boundary $\p V\cap\Lm$ in the target as well. 
The domains $U\subset\Om$ and $V\subset\Lm$ are called active domain and active target, respectively. 

In fact, $\nabla u$ is the optimal transport map from $f\chi_U+g(1-\chi_V)$ onto $g$ with a convex target $\Lm$, interior regularity and strict convexity of $u$ inside the active domain follows from Caffarelli's results \cite{C91, C92}.
The boundary regularity of $u$ is much harder to tackle, but is more attractive and has drawn lots of attention in recent years. 
An important fact is that the regularity of $u$ up to the free boundary $\p U\cap\Om$ implies the regularity of the free boundary itself in the sense that $\nabla u$ gives the direction of the normal to $\p U\cap\Om$.

Recall that for the complete transport problem, namely $m=\|f\|_1=\|g\|_1$, to show $u\in C^{1,\alpha}(\overline\Om)$ one needs both $\Om$ and $\Lm$ to be convex, see Caffarelli \cite{C92b}.
To show $u\in C^{2,\alpha}(\overline\Om)$ one further needs $\Om, \Lm$ to be uniformly convex with certain smoothness, see Caffarelli \cite{C96}, Delano\"e \cite{D91}, and Urbas \cite{U1}. 
Even though recently we reduced the uniform convexity assumption to convexity \cite{CLW1} (see also \cite{CLW2, SY} for dimension two and \cite{CLW3, J} for small perturbation of convex domains), the above boundary theory cannot be applied directly to the partial transport problem since $U$ and $V$ generally fail to be convex. 

For the partial transport problem \eqref{mass}, assuming $\Om$ and $\Lm$ are strictly convex and separated by a hyperplane, Caffarelli and McCann \cite{CM} proved that $u$ is $C^{1,\alpha}$ up to $\p U\cap\Om$, and thus obtained the $C^{1,\alpha}$ regularity of the free boundary $\p U\cap\Om$.  
When $\Omega$ and ${\Omega^*}$ are allowed to have overlap, Figalli \cite{AFi2,AFi} proved that away from the common region $\Omega\cap{\Omega^*},$ the free boundary is locally $C^1$, and this result was later improved by Indrei \cite{I} to local $C^{1,\alpha}$ away from the common region and up to a relatively closed singular set. 
 
\vskip5pt
By examining the proof in \cite{CM} carefully, the strict convexity condition on the domains can actually be removed.

\begin{theorem}\label{thm1}
Let $\Omega,$ ${\Omega^*}$ be two bounded, convex domains separated by a hyperplane. The densities $f, g$ are bounded away from zero and infinity. 
Let $m$ be the mass transported satisfying \eqref{mass}. Let $U\subset\Om$ and $V\subset\Lm$ be the active domain and target, respectively.
Then, there exists a constant $\alpha\in(0,1)$ such that the free boundary $\p U\cap\Om$ is $C^{1,\alpha}$ in the interior of $\Omega$.
\end{theorem} 
\begin{remark}
 \emph{This result is essentially due to Caffarelli and McCann \cite{CM}. In this paper we include some details to show how the strict convexity condition on the domains can be replaced by the usual convexity.
The approach also works for the case when the domains have overlap as considered in \cite{AFi2,AFi}, which leads to that the free boundary is $C^{1,\alpha}$ in the interior of $\Omega$ and away from the common region.}
\end{remark}
 
As mentioned by Caffarelli and McCann \cite{CM}, higher order regularity of the free boundary remains open in the partial transfer case, which indeed turns out to be extremely difficult.  
In this paper, we establish higher order regularity of the free boundary under the condition that $\Omega$ and ${\Omega^*}$ are sufficiently far apart.
As far as we know, this is the first result towards higher order regularity of free boundary in optimal partial transport problem.

\begin{theorem}\label{thm2}
In addition to the hypotheses of Theorem \ref{thm1}, assume that $\p\Om, \p\Lm\in C^{1,1}$ and denote $d:=\dist(\Omega, {\Omega^*})$.

\noindent $i)$ When $f, g\in C^0$, for any $\beta\in(0,1)$, there exists a constant $d_\beta$ such that if $d>d_\beta$, the free boundary $\partial U\cap \Omega$ is $C^{1,\beta}$. 

\noindent $ii)$ When $f, g\in C^{\alpha}$ for some $\alpha\in(0,1)$, there exists a constant $d_\alpha$ such that if $d>d_\alpha$, the free boundary $\partial U\cap \Omega$ is $C^{2,\alpha}$.

\noindent $iii)$ When $\Omega, {\Omega^*}, f, g$ are smooth, the free boundary is $C^\infty$ in the interior of $\Omega$, provided $d$ is sufficiently large.
\end{theorem}

\begin{remark}
\emph{
One can also consider the optimal partial transport problem with other cost functions. 
For example, when the cost $c=\frac{1}{p}|x-y|^p$ for some $p>1$, as $\dist(\Omega, {\Omega^*})\rightarrow \infty$,
the optimal transport map between active regions $U$ and $V$ will be close to the optimal transport map between some limiting domains 
$U_\infty$ and $V_\infty$ with the cost $\frac{1}{2}|x-y|^2$ (see \cite{CGN}). 
Our approach may also be adopted to study this problem. 
The main difference is that in general one has no $C^{1,\alpha}$ estimates of the free boundary priorly. 
In the special case when the cost function satisfies the Ma-Trudinger-Wang condition (A3) \cite{MTW}, a local $C^{1,\alpha}$ regularity of the free boundary was obtained in \cite{CI}.
We hope to investigate higher regularity for this problem in a separate work. 
}
\end{remark}

This paper is organised as follows. 
In \S\ref{S2} we introduce some useful notations and results in the optimal partial transport problem. 
In \S\ref{S3} we obtain the $C^1, C^{1,\alpha}$ regularities of the free boundary, and prove Theorem \ref{thm1}.
In \S\ref{S4} we obtain higher order regularities of the free boundary, and prove Theorem \ref{thm2}.
In \S\ref{S5} we introduce a related model of free boundary arising in an optimal transport problem, where the target contains two seperate parts. A more general version of this problem  was also investigated by Kitagawa and McCann in \cite{KM1}, where the $C^{1,\alpha}$ regularity of free boundary was also proved independently there.
 As an application of our argument, we also establish corresponding higher order regularities of the free boundary in this problem, see Theorem \ref{maint1}.

\begin{remark}
\emph{
Very recently in \cite{CLWfree}, when the target $\p\Lm\in C^2$ is uniformly convex, we are able to remove the ``far-apart" condition in Theorem \ref{thm2}, and obtain the $C^{2,\alpha}$ regularity of free boundary, provided the densities $f, g$ are $C^\alpha$ smooth. 
}
\end{remark}

 
 \section{Preliminaries and notations}\label{S2}
 In the following, we will always assume the densities $f, g$ are bounded from below and above by some positive constants.
For a fixed $m$ satisfying \eqref{mass}, it was shown in \cite{CM} that $\gamma_m,$ the minimiser of \eqref{mini}, is characterised by 
\beq\label{optmeas}
\gamma_m:=(Id \times T_m)_{\#}f_m=(T_m^{-1} \times Id)_{\#} g_m,
\eeq
 where $T_m$ is the optimal transport map from the active domain $U\subset \Omega$ to the active target $V\subset \Lm$, 
 the functions $f_m=f \chi_U$ and $g_m=g\chi_V$. 
 Indeed, $T_m=Du$ for some convex potential function $u$ solving
 \begin{equation} \label{push}
(Du)_{\#}(f_m+(g-g_m))= g
\end{equation}
with a convex target ${\Omega^*}.$
And by the interior regularity and strict convexity of $u$ \cite{C91, C92}, one has
\beq \label{homo int}
Du : U\to V \mbox{ is a homeomorphism between active interiors.}
\eeq

Similarly, $T_m^{-1}=Dv,$  for some convex function $v$ solving
$$ (Dv)_{\#}((f-f_m)+g_m) = f, $$
with a convex target $\Om$.
By \cite[Lemma 2]{C92} we can extend $u, v$ globally to $\R^n$ as follows
\beq\label{newv1}
\tilde u(x) = \sup\{L(x) : L \mbox{ affine, support of $u$ at some }x_0\in ({\Omega^*}\setminus V)\cup U\}
\eeq
\beq\label{newv}
\tilde v(x) = \sup\{L(x) : L \mbox{ affine, support of $v$ at some }x_0\in (\Om\setminus U)\cup V\}.
\eeq
For brevity, we still denote by $u, v$ the extensions $\tilde u, \tilde v$.  
Let $$v^*(x):=\sup_{y\in \mathbb{R}^n} x\cdot y-v(y),\ \text{for}\ x\in \bar{\Omega}$$
$$u^*(y):=\sup_{x\in \mathbb{R}^n} y\cdot x-u(x),\ \text{for}\ y\in \bar{{\Omega^*}}$$
be the standard Legendre transform of $u, v.$
The following two facts are very important for our argument:\\
1, $u(x)=v^*(x)$ for any $x\in U,$ $v(y)=u^*(y)$ for any $y\in V$,\\
2, $Dv(x)=x$ for a.e $x\in \Omega\setminus U,$ hence $v^*=1/2|x|^2+C$ on each connect component of $\Omega\setminus \bar{U}.$
Similarly, $u^*=1/2|x|^2+C$ on each connect component of ${\Omega^*}\setminus \bar{V}.$

Then, $v$ is a globally Lipschitz convex solution of
\beq\label{Asolv} 
	C_1(\chi_{\Om\setminus U}+\chi_V) \leq \det D_{ij}v \leq C_2 (\chi_{\Om\setminus U}+\chi_V),
\eeq
in the sense of Alexandrov, where $C_1, C_2$ are positive constants depending on the upper and lower bounds of $f, g$. 
 
In general, given a convex function $v : \R^n\to(-\infty, \infty]$ we define its associated \emph{Monge-Amp\`ere measure} $M_v$ on $\R^n$ by
\beq\label{MAmeas}
M_v(B) := \Vol[\p v(B)]
\eeq
for every Borel set $B\subset\R^n$. 
If $v$ is smooth and strictly convex, then
$$ M_v(B) = \int_B \det [D^2v(x)] \,dx. $$ 
The inequality \eqref{Asolv} is interpreted in the above measure sense, namely $\det D_{ij}v = f$ if
$$ M_v(B) = \int_B f $$
for every Borel set $B\subset\R^n$.
Hence, \eqref{Asolv} implies that the Monge-Amp\`ere measure $M_v$ is actually supported and bounded on $(\Om\setminus U)\cup V$. 

\vskip5pt 
Next, we recall the \emph{interior ball condition} obtained in \cite{CM}, which will be useful in our subsequent analysis. 
\begin{lemma}\label{intball}
Let $x\in U$ and $y=Du(x)$, then 
	\[ \Om\cap B_{|x-y|}(y) \subset U. \]
Likewise, let $y\in V$ and $x=Dv(y)$, then
	\[ \Lm\cap B_{|x-y|}(x) \subset V. \]
\end{lemma}
 
When $u$ is $C^1$ up to the free boundary $\p U\cap\Om$, one can see that \cite{CM} the unit inner normal of $\partial U\cap \Omega$ is given by
 \begin{equation}\label{normalformular}
 \nu(x)=\frac{Du(x)-x}{|Du(x)-x|},\quad \mbox{ at } x\in\p U\cap\Om.
 \end{equation}
Hence, the regularity of $u$ up to the free boundary $\partial U\cap \Omega$ implies the regularity of the free boundary itself.
By duality, the $C^{1,\alpha}$ regularity of $u$ up to $\partial U\cap \Omega$ actually follows from a quantified strict convexity of $v$ up to the fixed boundary $\p V\cap\p\Lm$, see \cite{C92b,CM}.

Useful elements in investigating the convexity and regularity of the convex function $v$ on the boundary are the centred sections and sub-level sets, see \cite{C92b,C96}.
\begin{definition}\label{defS}
Let $v$ be the above convex function, extended in \eqref{newv}. 
Let $y_0\in V$ and $h>0$ small. We denote 
\beq\label{sect}
	S^c_{h}[v](y_0) := \left\{y\in\R^n : v(y)< v(y_0) + (y-y_0)\cdot x+ h\right\}
\eeq
as the \emph{centred section} of $v$ with height $h$, where $x\in \R^n$ is chosen such that the centre of mass of $S^c_{h}[v](y_0)$ is $y_0$. 
Also, we denote
\beq\label{sub}
S_h[v](y_0) : =\left\{y\in V : v(y) < \ell_{y_0}(y) + h\right\}
\eeq
as the \emph{sub-level set} of $v$ with height $h$, where $\ell_{y_0}$ is a support function of $v$ at $y_0$. 
\end{definition}

\begin{remark}\label{reeq}
\emph{
If $v$ is strictly convex up to the boundary, we actually have an equivalency relation between its sub-level sets $S_h[v](y_0)$ and centred sections $S^c_{h}[v](y_0)$, that is for all small $h>0$,
\beq\label{equi0} 
S^c_{b^{-1}h}[v](y_0) \cap V \subset S_h[v](y_0) \subset S^c_{bh}[v](y_0) \cap V,
\eeq
where $b\geq1$ is a constant independent of $h$. 
For the proof of \eqref{equi0}, we refer the reader to \cite{C96} and \cite[Lemma 2.2]{CLW1}.
}
\end{remark}


\section{$C^1$ and $C^{1,\alpha}$ regularities}\label{S3}

Recall that in \cite[\S6]{CM} by assuming $\Lm$ is strictly convex, Caffarelli and McCann showed that $u$ is $C^1$ up to the free boundary $\p U\cap \Om$ in the sense that there exists $\tilde u\in C^1(\R^n)$ agrees with $u$ on $U\cap\Om$ and $D\tilde u(\R^n)=\overline\Lm$.  
Their idea was that by modifying the function $v$ to be $+\infty$ outside $\overline\Lm$, the desired extension $\tilde u$ is then the Legendre transform of $v$ satisfying $\p\tilde u(\R^n)\subset\overline\Lm$.  
If $\tilde u$ was not differentiable at some $x\in \R^n$, there would exist two different points $p_1, p_2\in\overline\Lm$ such that its Legendre transform $v$ coincides with an affine function on the segment $\ell=\overline{p_1p_2}\subset\overline\Lm$.
From the strict convexity of $\Lm$, the segment $\ell$ must contain an interior point of $\Lm$. This contradicts the strict convexity of $v$ inside $\Lm$, \cite{C92}. Hence $\tilde u\in C^1(\R^n)$. 
Under the strict convexity assumption on domains, in \cite[\S7]{CM} they further obtained $u$ is $C^{1,\alpha}$ up to the free boundary. 

In this section, we remove the strict convexity assumption on domains, and prove that $u\in C^{1,\alpha}$ up to the free boundary.

\begin{lemma}\label{C1lem}
Let $x_0\in\p U\cap\Om$. 
There exists a unique $p\in\overline V$ such that for any sequence $y_k \in U$ converging to $x_0$, namely $\lim_{k\to\infty} y_k =x_0$, one has
	$ \lim_{k\to\infty} Du(y_k) = p. $
\end{lemma} 
	
\begin{proof}
Without loss of generality we may assume that $x_0$ is the origin. 
Suppose to the contrary that there exist two sequences $\{y_k\}$ and $\{z_k\}$ in $U$ converging to the origin, but $Du(y_k)\to p_1$, $Du(z_k)\to p_2$ with $p_1\neq p_2$.
From \eqref{homo int}, one has $p_1, p_2\in\p V\cap \p\Lm$, and by duality $0\in\p v(p_1) \cap \p v(p_2)$. 
Subtracting a constant to the dual potential $v$, we have $v\geq 0= \min v$ and the segment
\beq\label{segin} 
\overline{p_1p_2} \subset \{v=0\} =: \mathcal{C}_0. 
\eeq
In order to derive a contradiction, we first observe some geometric properties of $\mathcal{C}_0$ in the following claims. 

\emph{Claim $\# 1$.} The contact set $\mathcal{C}_0$ is convex and bounded. 

\emph{Proof of claim.} Since $v$ is convex, the contact set $\mathcal{C}_0$ is convex as well. If $\mathcal{C}_0$ is not bounded, it will contain a ray $\{p_0+te : \forall\, t>0\}$ for some point $p_0\in\R^n$ and a unit vector $e\in\mathbb{S}^{n-1}$. 
Again by the convexity of $v$, one has at any point $p\in\R^n$, $\partial_ev(p)\leq 0$.
Hence, by duality, $\Om\subset\{x\in\mathbb{R}^n : x\cdot e\leq 0\}$, which contradicts with the assumption $0\in\Om$ is an interior point. 
\qed

\vskip5pt
Let $q_0\in\p V$ be an \emph{extreme} point of $\mathcal{C}_0$, namely $q_0\in\mathcal{C}_0$ cannot be expressed as a convex combination $q_0=(1-\lambda)q_1+\lambda q_2$ of points $q_1, q_2\in\mathcal{C}_0$ with $\lambda\in(0,1)$ unless $q_1=q_2$. 
Note that from \eqref{Asolv} and \cite{C92} extreme points of $\mathcal{C}_0$ can only lay on $\p\Om\cup\p V$. Then from \eqref{segin} and Claim $\# 1$, there must exist an extreme point $q_0\in\p V$ of $\mathcal{C}_0$.

\emph{Claim $\# 2$.} There is a sequence $q_i\in V$ such that $q_i\to q_0$ and $Dv(q_i)\to0$, as $i\to\infty$.

\emph{Proof of claim.} Since the extreme point $q_0\in\p V$, one has a sequence $q_i\in V$ converging to $q_0$. 
From the interior regularity, $v$ is differentiable at $q_i$ and $Dv(q_i)\in U$, for each $i=1,2,\cdots$. 
Suppose $Dv(q_i)\to z\in \overline U$, but $z\neq 0$. 
Then it implies $q_0\in\p v^*(0)\cap \p v^*(z)$, where $v^*$ is the Legendre transform of $v.$ Thus $v^*$ is affine on the segment $\overline{oz}$.
Note that the origin $0\in \p U\cap\Om$ is an interior point of $\Om$, the above contradicts to the fact that $v^*$ is strictly convex in $\Om$ as $v^*(x)=\frac12|x|^2+C$ in $\Om\setminus\overline U$ and $v^*=u$ is strictly convex in $U$. 
\qed

\vskip5pt
\emph{Claim $\# 3$.} There is a small constant $\delta_0>0$ such that $B_{\delta_0}(q_0)\cap\Lm\subset V$.

\emph{Proof of claim.} 
If this claim does not hold, then $q_0\in \overline{\Lm\cap\p V}$.
By Claim $\# 2$, let $x_i=Dv(q_i)\in U$, then $x_i\to0$. From Lemma \ref{intball}, $B_{|q_i-x_i|}(q_i) \cap\Om \subset U$ for each $i=1,2,\cdots$. Taking the limit $i\to\infty$ we have
	\[ B_{|q_0|}(q_0) \cap \Om \subset U. \]
Since $0\in\p U\cap\Om$, for $t>0$ sufficiently small the point $x_t:=tq_0\in U$. 
Let $q_t=Du(x_t)\in V$. 
By monotonicity of $\p u$, we have
\beq\label{conv1}
x_t\cdot(q_t-q_0) \geq 0.
\eeq
Again from Lemma \ref{intball}
\beq\label{ball2}
B_{|q_t-x_t|}(x_t) \cap \Lm \subset V.
\eeq	
By \eqref{conv1}, one can see that $|x_t-q_0|<|x_t-q_t|$, and thus $q_0\in B_{|q_t-x_t|}(x_t)$. Hence by \eqref{ball2}, any point $\tilde q\in\Lm$ sufficiently close to $q_0$ must also belong to $V$. This contradicts the assumption $q_0\in \overline{\Lm\cap\p V}$, and thus the claim is proved. 
\qed

\vskip5pt
Now, we go back to \eqref{segin} and denote 
$$ \hat p = \frac{p_1+p_2}{2}. $$
Note that we can actually assume that $q_0$ is \emph{exposed}, namely there is some hyperplane touches $\mathcal{C}_0$ only at $q_0$. This is due to the fact that exposed points are dense in the set of extreme points, and by \eqref{Asolv} exposed points can only lay on two separate domains either $\Om\setminus U$ or $V$. Hence, there must be an exposed point $q_0\in\p V$ of $\mathcal{C}_0$, and a unit vector $e_1$ such that
$$ (q-q_0)\cdot e_1 \geq 0,\quad\forall q\in\mathcal{C}_0$$
with the equality holds only at $q=q_0$.

Let 
$$ q_\delta :=(1-\delta)q_0 + \delta\hat p$$
for $\delta>0$ small. 
We may assume $q_\delta\in\p V\cap\mathcal{C}_0$, otherwise the proof is done by the interior strict convexity of $v$. 
Consider the centred section $S^c_\varepsilon(q_\delta)=S^c_\varepsilon[v](q_\delta)$ of $v$ at $q_\delta$ with height $\varepsilon$, defined in \eqref{sect}.
Denote by $\ell$ the straight line passing through $q_0$ and $q_\delta$, and intersects $\p S^c_\varepsilon(q_\delta)$ at two points $q_\varepsilon$, $\tilde q_\varepsilon$, namely 
$$ \ell\cap\p S^c_\varepsilon(q_\delta) = \left\{q_\varepsilon, \tilde q_\varepsilon \right\}. $$
Since $S^c_\varepsilon(q_\delta)$ is balanced at $q_\delta$, we may assume
$$(\tilde q_\varepsilon-q_\delta)\cdot e_1>0,\quad\mbox{while }\quad (q_\varepsilon-q_\delta)\cdot e_1<0. $$

Denote the half space 
$$ H^\delta := \{q\in\R^n : (q-q_0)\cdot e_1 \leq (q_\delta-q_0)\cdot e_1 \}. $$
Then we have the following observations:
\begin{itemize}
\item[$i)$] $v\leq C\varepsilon$ in $S^c_\varepsilon(q_\delta)$. 
\item[$ii)$] $S^c_\varepsilon(q_\delta) \cap H^\delta \to \mathcal{C}_0 \cap H^\delta$ in Hausdorff distance, as $\varepsilon\to0$.
\item[$iii)$] $q_\varepsilon \to q_0$ as $\varepsilon\to0$, hence
$$ \frac{|q_\varepsilon q_0|}{|q_\varepsilon q_\delta|} \to 0 \quad \mbox{ as }\varepsilon\to0. $$
\end{itemize}
Fix $\delta$ small enough so that $q_\delta$ is sufficiently close to the exposed point $q_0$. 
Since $S^c_\varepsilon(q_\delta)$ is balanced around $q_\delta$, from the above observations we can see that when $\varepsilon$ is small enough, 
$$ S^c_\varepsilon(q_\delta) \cap \Om = \emptyset. $$
Furthermore, thanks to Claim $\# 3$ we also have the \emph{doubling property} for small $\varepsilon$,
\beq\label{doubling} 
M_v[\frac{1}{2}S^c_\varepsilon(q_\delta)] \geq \beta M_v[S^c_\varepsilon(q_\delta)],
\eeq
where $\beta>0$ is a universal constant, and $M_v$ is the Monge-Amp\`ere measure in \eqref{MAmeas}.

Finally, by using a similar argument of Caffarelli's in \cite{C92b} we can derive a contradiction if \eqref{segin} occurs. 
Denote by $L$ the affine function determining $S^c_\varepsilon(q_\delta)$, namely
	$$ S_\varepsilon := S^c_\varepsilon(q_\delta) = \{x\in R^n : v(x)<L(x)\}. $$
Let $w:=v-L$. Then
\begin{equation*}
\left\{\begin{array}{cl}
\det\,D^2w \approx \chi_{S_\varepsilon\cap V} \quad & \mbox{ in } S_\varepsilon \\
w=0 \quad & \mbox{ on }\p S_\varepsilon,
\end{array}
\right.
\end{equation*}
where $A\approx B$ means $C^{-1}B\leq A\leq CB$ with a universal constant $C>0$. 
Note that $w$ is affine on the segment $\overline{q_0\tilde q_\varepsilon}$, $w(\tilde q_\varepsilon)=0$ and $w(q_\delta)=-\varepsilon$, thus $w(q_0) \leq -\varepsilon$.

By normalisation
$$ w' (q) = \frac{1}{\varepsilon} w(T^{-1}q), $$ 
where $T$ is a linear transformation such that $S'_\varepsilon=T(S_\varepsilon)  \sim B_1$, one has $w'$ satisfies
\begin{equation*}
\left\{\begin{array}{rll}
\det\,D^2w' &\!\! \approx \chi_{S'_\varepsilon\cap V'}/(|T|^2\varepsilon^n) \quad & \mbox{ in } S'_\varepsilon \\
w' &\!\! =0 \quad & \mbox{ on }\p S'_\varepsilon,
\end{array}
\right.
\end{equation*}
and $\det\,D^2w'$ satisfies the doubling property \eqref{doubling} in $S'_\varepsilon$. 
On one hand, $|\inf w'| \approx |w'(q'_\delta)| =1$, and $|w'(q'_0)|\geq1$. 
On the other hand,
\begin{equation*} 
\begin{split}
|w'(q'_0)|^n &\lesssim |\inf w'|^n\, d(q'_0, \p S'_\varepsilon)  \\
	&\leq C|q'_\varepsilon q'_0| = C\frac{|q_\varepsilon q_0|}{q_\varepsilon q_\delta} \to 0 \quad \mbox{ as }\varepsilon\to0.
\end{split}
\end{equation*}
This contradiction proves Lemma \ref{C1lem}, namely $u$ is $C^1$ up to the free boundary $\p U\cap\Om$. 
\end{proof}

\begin{corollary}\label{coro1}
The free boundary $\p U\cap\Om$ is $C^1$ in the interior of $\Om$.  
\end{corollary}

\begin{proof}
For any $\delta>0$, let $\Om_\delta=\{x\in\Om : \dist(x,\pom)>\delta\}$. From Lemma \ref{C1lem}, $u\in C^1(\overline{U\cap\Om_\delta})$. 
By the Whitney extension theorem \cite{Ho}, there exists a function $\tilde u\in C^1(\mathbb{R}^n)$ such that $\tilde u=u$ and $D\tilde u=Du$ on $\overline{U\cap\Om_\delta}$. 
Recall that $u$ is the minimiser of \eqref{dual mini}, and at $x\in\partial U\cap\Om$, $u(x)=h(x)=(|x|^2-\lambda)/2$ for some constant $\lambda$. 
Since at $x\in\partial U\cap\Om$, $D h(x)=x$, while $D \tilde u(x)\in \p\Lm$, we have 
$$ |D(h-\tilde u)|\neq0 \quad\mbox{on } \partial U\cap\Om.$$
By the implicit function theorem, $\partial U\cap\Om$ is locally a $C^1$ hypersurface, and thus we obtain that the free boundary $\p U\cap\Om_\delta$ is $C^1$. 
\end{proof}

We are now ready to prove Theorem \ref{thm1}.

\begin{proof}[Proof of Theorem \ref{thm1}]
Recall that free boundary never maps to free boundary \cite{CM}, which implies that $\p U\cap\Om$ is mapped to a part of the fixed boundary $\p V\cap\p\Lm$. Hence, showing $v$ is $p$-uniformly convex \cite{CM} up to the image of the free boundary $\p U\cap\Om$ implies that $u$ is $C^{1,\alpha}$ up to $\p U\cap\Om$, and thus the free boundary $\p U\cap\Om$ itself is $C^{1,\alpha}$ as well. 
 
Let $0\in\p U\cap\Om$. From Lemma \ref{C1lem}, there is a unique $q_0\in\overline V$ such that $q_0=\lim_{k\to\infty}Du(x_k)$ for any $x_k\in U$ and $x_k\to0$ as $k\to\infty$. From the proof of Claim $\# 3$ in Lemma \ref{C1lem}, we also have $B_{\delta_0}(q_0)\cap\Lm\subset V$ for some small $\delta_0>0$. 
In order to show $v$ is $p$-uniformly convex in $B_{r_0}(q_0)\cap V\cap \Lm$ for a small $r_0>0$, we need the following localisation property:
\begin{lemma}\label{1loc}
Let $S^c_\varepsilon(q)$ be the centred section of $ v$ given in \eqref{sect}. 
Then,
\begin{itemize}
\item[(i)] $S^c_\varepsilon(q_0)\cap\Lm\subset V$, $S^c_\varepsilon(q_0)\cap\bom=\emptyset$, for $\varepsilon>0$ small; and
\item[(ii)] $\exists$ $\eta_0>0$ small, such that $\forall q\in B_{\eta_0}(q_0)$, $S^c_\varepsilon(q)\cap\Lm\subset V$, $S^c_\varepsilon(q)\cap\bom=\emptyset$, for $\varepsilon>0$ small enough. 
\end{itemize}
\end{lemma}

\begin{proof}
First, we prove (i). Suppose to the contrary that there exists a sequence $\{\varepsilon_k\}$ converging to $0$ as $k\to\infty$, such that for each $k$ there is a segment $I_k\subset S^c_{\varepsilon_k}(q_0)$ passing through $q_0$ and balanced around $q_0$, and $|I_k|\geq b_0$ for a positive constant $b_0$.
Taking $k\to\infty$, we have $I_k\to I_0\subset\{ v=0\}$ for some $I_0$ balanced around $q_0$ and $|I_0|\geq b_0$. 
Hence, $q_0$ cannot be an extreme point of $\{ v=0\}$.
From the proof of Lemma \ref{C1lem}, the set $\{v=0\}$ must be bounded, and the extremal points of $\{v=0\}$ are contained in $\bom\cup\overline V$.
Therefore, $\{v=0\}\cap\p V$ contains at least one extreme point of $\{v=0\}$, and denoted it by $q_1\neq q_0$. 
However, this contradicts with the $C^1$ regularity of $u$ in Lemma \ref{C1lem}.

Next, we prove (ii). Note that the proof of (i) implies that for any small $r_0>0$, there exists $\varepsilon_0>0$ such that $S^c_{\varepsilon_0}(q_0)\subset B_{r_0}(q_0)$.
Let $\overline{q_1 q_2}$ be any segment passing through $q$ and ending on $\p S^c_{\varepsilon_0}(q_0)$, namely $q_1,q_2\in \p S^c_{\varepsilon_0}(q_0)$.
For $\eta_0$ sufficiently small, we have $v(q_1)+v(q_2) \gtrsim \varepsilon_0$.
Without loss of generality we may assume $v(q_1)\gtrsim \varepsilon_0$, and thus $q_1\notin S^c_\varepsilon(q)$ for $\varepsilon \ll \varepsilon_0$ sufficiently small. 
Denote by $\ell$ the line passing through $q_1, q_2$. Let 
$$ \left\{\tilde q_1, \tilde q_2 \right\} := \ell \cap \p S^c_\varepsilon(q) $$
be the intersections of $\ell$ and $ \p S^c_\varepsilon(q) $, while $q_i, \tilde q_i$ are on the same side of $q$, namely $(q_i-q)\cdot(\tilde q_i-q)>0$, for $i=1,2$. Then, since $q_1\notin S^c_\varepsilon(q)$,
$$ |\tilde q_1 - q| < |q_1-q| \lesssim \diam\left(S^c_{\varepsilon_0}(q_0)\right) \lesssim r_0.$$
Since $S^c_\varepsilon(q)$ (and hence $\ell \cap \p S^c_\varepsilon(q)$) is balanced around $q$, we also have
$$  |\tilde q_2 - q|  \lesssim r_0. $$
By the arbitrariness of small $r_0$, we conclude the proof.  
\end{proof}
Once having the above localisation property, following \cite[\S7]{CM} we can prove the doubling property \eqref{doubling} and obtain the $p$-uniform convexity of $v$ via the approach of geometric decay of centred sections \cite{C92b,C96}. By duality, we have $u$ is $C^{1,\alpha}$ near the origin $0\in\p U\cap\Om$.

A final step is to show the free boundary $\p U\cap\Om$ is $C^{1,\alpha}$ in the interior of $\Om$.
From Lemma \ref{C1lem} and Corollary \ref{coro1}, we have \eqref{normalformular}, that is the inner normal of the free boundary at $x\in \p U\cap\Om$ is given by $Du(x)-x$, which is thus H\"older continuous in the interior of $\Om$. 
Therefore, Theorem \ref{thm1} is proved. 
\end{proof}

\begin{remark}\label{nonconvex}
Note that for the regularity of free boundary $\partial U\cap \Omega$, we can 
actually remove the convexity condition on $\Omega$ in Theorem \ref{thm1}. The reason is as follows:
 Recall that $v^*=|x|^2/2$ in $\Omega\setminus U$,
and $v^*=u$ in $U$ is strictly convex in the interior of $U,$ hence $v^*$ is strictly convex in $\Omega.$ Fix
$x_0\in \partial U\cap \Omega,$ denote $y_0=Du(x_0)\in \partial \Omega^*.$ Since $v^*$ is strictly convex near $x_0,$ the function $v$ is $C^1$ near $y_0.$  Let $\Omega_1=B_{r}(x_0),\ \Omega_1^*=Du(B_r(x_0))$ for some small $r>0,$ we have $Du$ solves an optimal partial transport problem from $\Omega_1$ to $\Omega_1^*.$ Now, $\Omega_1$ is convex and $\Omega_1^*$ is locally convex near $y_0,$ hence we can apply the previous proof to show that $u$ is $C^{1,\alpha}$ near $x_0.$  
For the same reason, for the higher order regularity of $\partial U\cap\Omega$, we can also remove the convexity condition on $\Omega$ in \cite[Theorem 1.1]{CLWfree}.
\end{remark}

\section{Higher order regularities} \label{S4}

Let $\Omega, {\Omega^*}$ be bounded, convex domains with $C^{1,1}$ boundaries. 
When $d:=\dist(\Omega, {\Omega^*})$ is sufficiently large, heuristically one can see that $\frac{y-x}{|y-x|}$ is uniformly close to a unit vector $e\in\mathbb{S}^{n-1}$ for all $x\in \Omega$, $y\in {\Omega^*}$.
Without loss of generality we may assume $e=e_n=(0,\cdots,0,1)$.
Let $f, g$ be the densities supported on $\Om, \Lm$, respectively, and $m$ is the partial mass satisfying \eqref{mass}.
There are two constants $a, b\in\mathbb{R}$ such that 
 $$m=\int_{\{x^n>a\}\cap \Omega}f = \int_{\{y^n<b\}\cap {\Omega^*}}g.$$
Let $x\in\p U\cap\Om$, and $\nu(x)$ be the unit inner normal of the free boundary $\p U\cap\Om$ at $x$.
We have $|\nu(x)-e_n|$ is as small as we want, provided $d$ is large enough, thus 
 \begin{equation}\label{gq1}
 \{x^n>a+\delta\}\cap \Omega\subset U \subset \{x^n>a-\delta\}\cap \Omega,
 \end{equation}
 and similarly
\begin{equation} \label{gq2}
 \{y^n<b-\delta\}\cap {\Omega^*}\subset V \subset \{y^n<b+\delta\}\cap {\Omega^*},
 \end{equation}
  where $\delta$ can be as small as we want, provided $d$ is large enough.
  By translating the coordinates, we may assume $b=0$ and $0\in \p V\cap {\Omega^*}.$
  
Let $u$ be the potential function satisfying \eqref{push}--\eqref{homo int}. 
It is straightforward to check that $\tilde{u}(x):=u(x+ae_n)$ solves the optimal transport problem from
  $(\tilde{U}, \tilde{f}\chi_{\tilde{U}})$ to $(V, g\chi_V),$ where $\tilde{U}:=\{x: x+ae_n\in U\},$ and $\tilde{f}(x):=f(x+ae_n).$
By \eqref{gq1} and the above sliding, we have
   \begin{equation}\label{closeto1}
 \{x^n>\delta\}\cap \Omega\subset U \subset \{x^n>-\delta\}\cap \Omega,
 \end{equation}
 where $\delta$ can be as small as we want, provided $d$ is large enough.
Note that for simplicity, in \eqref{closeto1} and below we still use $u, f, U, \Om$ to denote the items after the sliding.  
 
Now, we consider the limit case.  
Let $U_\infty:=\{x^n>0\}\cap \Omega$, $V_\infty:=\{y^n<0\}\cap {\Omega^*}$. 
Let $u_\infty$ be the convex function solving
\beq\label{lim1}
(Du_\infty)_{\#} {f}\chi_{U_\infty}=g\chi_{V_\infty},
\eeq
with $u_\infty(x_0)={u}(x_0)$ for some $x_0\in U_\infty\cap U$.
By a standard compactness argument, we have $u_\infty$ and $u$ will be as close as we want, provided $d$ is large enough. 
In the spirit of this observation, we shall obtain the regularity of $u$ as a small perturbation of $u_\infty$, but for convenience and consistency of notations with \S3, we shall work with dual potentials $v, v_\infty$ instead. In \S\ref{s41}, we obtain the smoothness of $v_\infty$ in the limit case. In \S\ref{s42} and \S\ref{s43}, by compactness and perturbation argument, we show that $v$ is sufficiently close to $v_\infty$, and further to a quadratic function, when $d$ is large enough. Theorem \ref{thm2} is then proved in \S\ref{s44}.

\subsection{Smoothness in the limit case}\label{s41}

Let $v_\infty$ be the dual potential of $u_\infty$, namely
\beq\label{lim2}
(Dv_\infty)_{\#} g\chi_{V_\infty} = {f}\chi_{U_\infty}.
\eeq
We shall obtain the regularity of $v_\infty$ up to the boundary $\p V_\infty\cap\Lm$.
Recall that when the domains are uniformly convex, the boundary regularity of potential function was obtained by Caffarelli \cite{C92b,C96}, Delano\"e \cite{D91}, and Urbas \cite{U1}. 
The uniform convexity assumption on domains has been relaxed in our recent work \cite{CLW1}, see also \cite{CLW2,SY} for dimension two, and \cite{CLW3} for perturbations of convex domains. 
However, none of the above applies directly to the limit case \eqref{lim1}--\eqref{lim2}, because the domains $U_\infty, V_\infty$ are neither uniformly convex nor even $C^1$ smooth. 
In the following we will overcome this obstacle by deriving key estimates that ensure the boundary regularity of $v_\infty$ up to $\p V_\infty\cap\Lm$.

\begin{lemma}[Obliqueness]\label{obliqueness}
$\forall$ $y_0\in \p V_\infty\cap \Lm$, let $x_0=Dv_\infty(y_0)$. Then, $\nu(x_0)\cdot (-e_n)>0,$ and $x_0^n>0$, where $\nu(x_0)$ is the unit inner normal of $\p U_\infty$ at $x_0$. 
\end{lemma}
\begin{proof}
First, we show that $\nu(x_0)\cdot (-e_n)\geq 0$. 
Suppose to the contrary that $\nu(x_0)\cdot (-e_n)<0$. 
Since $\pom\in C^{1,1}$, we have $x_t:=x_0+te_n\in \Omega\cap \{x^n>0\}$ for small $t>0$. By monotonicity, 
$$(x_t-x_0)\cdot (Du_\infty(x_t)-y_0)\geq 0,$$ 
which implies $Du_\infty(x_t)\in \{y^n=0\},$ namely
 $Du_\infty(x_t)$ is on the boundary of $V_\infty.$ 
 This contradicts to the fact that $x_t$ is an interior point of $U_\infty$, and $Du_\infty : U_\infty \to V_\infty$ is a homeomorphism, see \eqref{homo int}. 

Next, we show that $\nu(x_0)\cdot (-e_n)>0$. 
By a translation (preserving the $y_n$-axis) and subtracting an affine function, we may assume that $y_0=0$ and $v_\infty \geq 0=v_\infty (0)$. 
Simultaneously, by duality this makes $x_0=0$ and $u_\infty\geq u_\infty(x_0)$. By subtracting a constant, we also assume that $u_\infty(x_0)=0$. 
Suppose to the contrary that $\nu(x_0)\cdot(- e_n)=0$. Without loss of generality, we can assume $\nu(x_0)=e_1$, thus
\beq\label{newdd} 
	U_\infty \subset \{x\in\R^n : (x-x_0)\cdot e_1\geq0\}. 
\eeq
By duality, $D_{y_1}v_\infty \geq 0$ in $V_\infty$.
Let $y_t=-te_1\in  \partial V_\infty\cap\Lm$ for $t>0$ small.  
Then $v_\infty(y_t)\equiv0$, $\forall\, t\in[0,t_0]$ for some small constant $t_0$. 
Let $x_t=Dv_\infty(y_t)$. As $v_\infty\geq0$, by duality
$$ x_t \in \p U_\infty\cap\{x\in\R^n : (x-x_0)\cdot e_n<0\}. $$
Otherwise, for a small $\varepsilon>0$, $v_\infty(y_t-\varepsilon e_n)\leq 0$, which would contradict with the strict convexity of $v_\infty$ in the interior of $V_\infty$. 
Again, by monotonicity we have 
$$(Dv_\infty(y_t)-x_0)\cdot (y_t-y_0)\geq 0,$$ 
which implies $(x_t-x_0)\cdot e_1\leq0$. 
Then by \eqref{newdd}, 
$$ (x_t-x_0)\cdot e_1=0,\quad \forall\,t\in[0,t_0], $$
which implies that $\p U_\infty$ contains a flat segment $\overline{x_0 x_{t_0}}$ lies on the $x_1$-axis.
On the other hand, by duality $Du_\infty(x_t)=y_t$ lies on the $y_n$-axis, one has $Du_\infty(x_t)\cdot e_1=0$ for all $t\in[0,t_0]$.
Hence, $u_\infty(x_t)\equiv0$ for all $t\in[0,t_0]$. This contradicts the strict convexity of $u_\infty$ up to the fixed boundary $\p U_\infty\cap\pom$.

Lastly, we only need to rule out the case ``$\nu(x_0)\cdot (-e_n)>0$, $x_0^n=0$".  
If this occurs, by a rotation we may assume $\nu(x_0)\in span\{e_1,e_n\}$ and $\nu(x_0)\cdot e_1<0$. 
(Indeed, if $\nu(x_0)\cdot e_1=0$, the hyperplane $\{x_n=0\}$ would be tangential to $\pom$ and then $|U_\infty|=0$.)
Hence, there are two constants $\alpha, \beta > 0$ such that
$$ \nu(x_0) = -\alpha e_1 - \beta e_n. $$
Then, $\forall$ $x\in \overline U_\infty$ ($x\ne x_0$), we have
\begin{equation*} 
\begin{split}
(x-x_0)\cdot e_1 &= -\frac{1}{\alpha}(x-x_0)\cdot \nu(x_0) - \frac{\beta}{\alpha}(x-x_0)\cdot e_n  \\
	& \leq 0,
\end{split}
\end{equation*}
where the last inequality holds because of $(x-x_0)\cdot \nu(x_0)\geq0$ as $\nu(x_0)$ is the unit inner normal, and $(x-x_0)\cdot e_n \geq 0$ as $U_\infty\subset \{x^n>0\}$.
Let $y_t=y_0+te_1 \in \p V_\infty\cap\Lm$ for a small constant $t>0$. We can then make the same contradiction as in the last paragraph. 
Alternatively, since $y_0$ is a relatively interior point of free boundary $\p V_\infty\cap\Lm$, we can also exclude this case as in the proof of Claim 3 in Lemma \ref{C1lem}.
\end{proof}

In order to derive the boundary regularity of $v_\infty$ on $\p V_\infty\cap\Lm$, another important ingredient is the \emph{uniform density property}, 
namely there exists a positive constant $c_0$ such that for $y\in\p V_\infty\cap\Lm$ and $h>0$ small, the centred section $S^c_h(y)=S^c_h[v_\infty](y)$ satisfies 
\begin{equation}\label{unid}
\frac{\Vol(S^c_h(y)\cap V_\infty)}{\Vol(S^c_h(y))} \geq c_0.
\end{equation}
Note that since $\partial V_\infty\cap\Lm$ is flat, the uniform density property \eqref{unid} is trivial in this case. 
Once having the uniform density and obliqueness estimates, following the approach as in \cite{CLW1}, we can obtain the regularity of $v_\infty$ up to the flat boundary.

\begin{lemma}\label{limsm}
Let $U_\infty, V_\infty$ be the active regions, and $f, g$ be the densities in the above limit case.
Let $v_\infty$ be the convex function solving $(Dv_\infty)_{\#} g\chi_{V_\infty} = {f}\chi_{U_\infty}$.  Then,

\noindent a) if $f, g\in C^0$, then $v_\infty\in C^{1, \beta}(V_\infty)$ up to $\p V_\infty\cap\Lm$, for all $\beta\in (0,1)$;

\noindent b) if $f, g\in C^\alpha$, then $v_\infty\in C^{2,\alpha}(V_\infty)$ up to $\p V_\infty\cap\Lm$.
\end{lemma}

As a corollary of Lemma \ref{limsm}, we give some geometric properties of sub-level sets of $v_\infty$, which will be useful in the subsequent perturbation argument. 

Let $y_0\in \p V_\infty\cap{\Omega^*}$, and $x_0:= Dv_\infty(y_0)\in\pom$. 
By a similar approach as for Lemma \ref{limsm}, we can also obtain the regularity of $u_\infty$ at $x_0$, which then implies its Legendre transform $v_\infty$ is strictly convex at $y_0$. 
By a translation of coordinates, we may assume $x_0=y_0=0$.
From the strict convexity of $v_\infty$, by Remark \ref{reeq} we have an equivalency relation between its sub-level sets $S_{v_\infty, h}=S_h[v_\infty](0)$ and centred sections $S^c_{v_\infty,h}=S^c_h[v_\infty](0)$, namely for all small $h>0$
\beq\label{equi}
S^c_{v_\infty, b^{-1}h}\cap V_\infty \subset S_{v_\infty, h} \subset S^c_{v_\infty, bh}\cap V_\infty,
\eeq
where $b\geq1$ is a constant independent of $h$. 
The following lemma contains a more delicate geometric property of $S_{v_\infty, h}$ for $h>0$ small. 

\begin{lemma} \label{loc11}
Assume the densities $f, g$ are positive and continuous. 
For any $h>0$ small, there is a symmetric matrix $A$ satisfying $\det\, A=1$, and $\|A\|,\|A^{-1}\|\leq Ch^{-\epsilon}$
 such that 
\begin{equation}\label{gs1}
B_{C^{-1}h^{1/2}}\cap AV_\infty\subset AS_{v_\infty,h}\subset B_{Ch^{1/2}}\cap AV_\infty,
\end{equation}
and 
\begin{equation}\label{gs2}
B_{C^{-1}h^{1/2}}\cap A^{-1}U_\infty\subset A^{-1}Dv_\infty(S_{v_\infty, h})\subset B_{Ch^{1/2}}\cap A^{-1}U_\infty,
\end{equation}
where $C>0$ is a universal constant, and $\epsilon>0$ can be as small as we want.
\end{lemma}

\begin{proof}
Let $S_{v_\infty, h}^c$ be the centred section of $v_\infty$ with height $h$. 
From \cite{C96,CLW1}, $\Vol(S_{v_\infty, h}^c)\approx h^{n/2}$.
From Lemma \ref{limsm}-$a)$, for any small $\epsilon>0$, there is a symmetric matrix $A$ with $\det\, A=1$ and $\|A\|,\|A^{-1}\|\leq Ch^{-\epsilon}$ normalising $S_{v_\infty, h}^c$ such that
\beq\label{bb1}
	B_{C^{-1}h^{1/2}} \subset AS_{v_\infty, h}^c \subset B_{Ch^{1/2}}.
\eeq
From \eqref{equi} and the proof of Lemma \ref{obliqueness}, one has
\beq\label{bb2} 
	AS_{v_\infty, h} \sim A\left(S^c_{v_\infty, h}\cap V_\infty\right). 
\eeq
Hence, \eqref{gs1} follows from \eqref{bb1} and \eqref{bb2}. 

Next, we prove \eqref{gs2}. Let $u_\infty$ be the Legendre transform of $v_\infty$. 
Since $S^c_{v_\infty, h}$ is conjugate to $S^c_{u_\infty, h}$ (see \cite{C96,CLW1}), we have
	$$B_{C^{-1}h^{1/2}}  \subset A^{-1}S_{u_\infty, h}^c \subset B_{Ch^{1/2}} .$$
Therefore, in order to obtain \eqref{gs2}, it suffices to show
	\beq\label{ppc2}
	 	Dv_\infty(S_{v_\infty, h}) \sim S_{u_\infty, h}^c \cap U_\infty. 
	\eeq
	
As the inclusion relation is preserved under linear transformation, it will be convenient to prove \eqref{ppc2} in a normalised picture.
Making the transformation:
\begin{align*}
&V_\infty \rightarrow AV_\infty;\ U_\infty\rightarrow A^{-1}U_\infty; \\
& v_\infty(y)\rightarrow v_\infty(A^{-1}y);\ u_\infty(x)\rightarrow u_\infty(Ax),
\end{align*}
 we may assume that
 \begin{equation}\label{inc3}
  S^c_{v_\infty, h}\sim  B_{h^{1/2}}\sim  S^c_{u_\infty, h}.
  \end{equation}
From \eqref{equi}, we have \eqref{ppc2} is equivalent to $Dv_\infty(B_{h^{1/2}}\cap V_\infty) \sim B_{h^{1/2}}\cap U_\infty$, namely
\beq\label{ppc3}
	B_{C^{-1}h^{1/2}}\cap U_\infty \subset Dv_\infty(B_{h^{1/2}}\cap V_\infty) \subset B_{Ch^{1/2}}\cap U_\infty.
\eeq
The first inclusion in \eqref{ppc3} is due to the convexity of $v_\infty$. 
The second inclusion follows from a quantified strict convexity of $v_\infty$, namely there is a universal constant $\beta\in(0,1)$ such that
\beq\label{ppc5}
 S_{v_\infty, h/2} \subset \beta S_{v_\infty, h}.
\eeq
In fact, \eqref{ppc5} implies that $Dv_\infty(S_{v_\infty, h/2}) \subset B_{\frac{1}{2(1-\beta)}h^{1/2}}\cap U_\infty$. 
Hence, we have
$Dv_\infty(B_{h^{1/2}}\cap V_\infty) \subset B_{Ch^{1/2}}\cap U_\infty$ for a universal constant $C>0$ independent of $h$. 
By rescaling back, we then obtain \eqref{gs2}.
\end{proof}

\subsection{Closeness to the limit}\label{s42}

Now, we come back to our original problem. 
In order to show that when $d=\dist(\Om,\Lm)$ is sufficiently large, the original optimal transport problem is a small perturbation of the limit case considered in \S4.1, we shall prove that the active regions $U, V$ are small perturbations of $U_\infty, V_\infty$, and the potentials $u, v$ are small perturbations of $u_\infty, v_\infty$ as well. 

\subsubsection{Closeness of active regions}
We first prove that $U, V$ converge to $U_\infty, V_\infty$ in a $C^{1,\alpha}$ manner as $d\to\infty$. In other words, the free boundaries $\p U\cap\Om, \p V\cap\Lm\in C^{1,\alpha}$ become flatter and flatter, and their $C^{1,\alpha}$ norms are uniformly bounded. 
The exponent $\alpha\in (0,1)$ is the same as in the $C^{1,\alpha}$ result by Caffarelli and McCann \cite{CM}, or equivalently, in Theorem \ref{thm1}. 
The key point is to show that the $C^{1,\alpha}$ estimate is independent of $d$ as $d\to\infty$.

Recall that in \cite[\S7]{CM}, the $C^{1,\alpha}$ regularity of $v$ follows from a quantified, $p$-uniform convexity of $u$. 
Moreover, the exponent $\alpha$ and the $C^{1,\alpha}$ norm of $v$ essentially depend on the modulus of strict convexity of $u$, see \cite[Theorem 7.13]{CM}.
Hence, it suffices to obtain a uniform modulus of strict convexity of $u$ independent of $d$ as $d\to\infty$.

\begin{lemma}\label{loca91}
Assume $d$ is sufficiently large. 
Let $x_0\in\p U\cap \pom$ be a preimage of $y_0\in\p V\cap\Lm$, 
and $R>0$ such that $B_{2R}(x_0)$ contains no preimage of $\p(\p V\cap\Lm)$.
Then, there exists a small constant $h_0$ independent of $d$ such that $\forall$ $h<h_0$,
\beq\label{uni incl}
	S^c_{u,h}(x_0)\subset B_R(x_0).
\eeq
Moreover, $S^c_{u,h}(x_0)\cap U$ is convex, provided $h<h_0'$ for another constant $h_0'$ depending further on $m, R$, but independent of $d$.
\end{lemma}

Note that if $x$ is the preimage of an interior point $y\in \p V\cap\Lm$, the radius $R$ of the small neighbourhood $B_{2R}(x)$ only depends on the inner and outer radii of $V$, which in turn depend only on $\Lm$ and $m$ (the mass been transported), when $d$ is sufficiently large, 

\begin{proof}
By subtracting an affine function, we may assume that $u\geq 0$, $u(x_0)=0,$ and $Du(x_0)=0$, that simultaneously makes $y_0=0\in\p V\cap\Lm$.  
Since $S^c_{u,h}(x_0)$ is balanced about $x_0$, we have $u(x)\leq C_1h$, $\forall$ $x\in S^c_{u,h}(x_0)$, where $C_1>0$ is a universal constant. 

By Corollary \ref{coro1}, we can find a ball $B_{\hat r}(\hat y)\subset V$ for some $\hat y\in V$, where $\hat r:=|\hat y|$.
Then, for any $ e' \in \{e\in\mathbb{S}^{n-1} : e\cdot \hat y\geq0\}$, there exists a point $y'\in B_{\frac{1}{2}\hat r}(\hat y)$ such that $\hat e \cdot y'\geq \hat r/2$.
Since $y'\in V$, there exists a point $x'\in U$ such that 
$(x-x')\cdot y'+u(x')$ is the support plane of $u$ at $x'$. 
Therefore,
\begin{equation*}
\begin{split}
u(x_0+t e')&\geq (x_0+t e'-x')\cdot y'+u(x')\\
&\geq (x_0-x')\cdot y'+u(x')+t e'\cdot y'\\
&\geq (x_0-x')\cdot y'+\frac12t\hat r. 
\end{split}
\end{equation*}
From \S\ref{S3}, $v$ is $C^1$ up to the free boundary $\p V\cap\Lm$, and thus $|x_0-x'|\leq C_2 |y'| = 2C_2\hat r$, for a constant $C_2$ depending on the diameter of $U$. Hence,
\begin{equation*}
 u(x_0+t e') \geq -2C_2\hat r^2 + \frac12t\hat r  > C_1h
\end{equation*}
provided 
$$t\geq \frac{2C_2\hat r^2+C_1h}{\hat r/2}:=R_1.$$
Hence, $x_0+t e'\not\in S^c_{u,h}(x_0)$ for any $ e'\in \{e\in\mathbb{S}^{n-1} : e\cdot \hat y\geq0\}$, provided $t>R_1$. 
Since $S^c_{u,h}(x_0)$ is balanced around $x_0,$
we have $x_0+te\not\in S^c_{u,h}(x_0)$ for any $e\in\mathbb{S}^{n-1}$, provided $t>C_3R_1=R$. 
By tracing it back, for a given $R>0$ satisfying the assumption, letting 
$$\hat r < \frac{R}{8C_2C_3},\quad \mbox{ and }\quad h_0 := \frac{R^2}{32C_1C_2C_3^2}, $$
we can obtain that $S^c_{u,h}(x)\subset B_R(x)$, for any $h<h_0$.

Note that when $d$ is sufficiently large, $\delta$ will be sufficiently small in \eqref{gq1}--\eqref{gq2}. 
By the assumption that $B_{2R}(x_0)$ contains no preimage of $\p(\p V\cap\Lm)$, one can see that $2r_{x_0}:=\dist(x,\p U\cap\Om)>r_0$ for some positive constant $r_0$ depending on $m, R$.
Let $e\in\mathbb{S}^n$ such that $x_0+2r_{x_0}e\in\p U\cap\Om$. 
Then $B_{r_{x_0}}(x_0)$ does not touch the free boundary $\p U\cap\Om$, neither the preimage of $\p(\p V\cap\Lm)$. 
Therefore, from the proof of \eqref{uni incl}, there exists a constant $h_0'$ independent of $d$ such that $S^c_{u,h}(x_0)\subset B_{r_{x_0}}(x_0)$, and $S^c_{u,h}(x_0)\cap U=S^c_{u,h}(x_0)\cap \Omega$ is convex, for all $h<h_0'$. 
\end{proof}

\begin{remark}\label{uniC1alpha}
\emph{
From Lemma \ref{loca91} and the argument in \cite[\S7]{CM}, $u$ is $p$-uniformly convex on $x\in\p U\cap\p\Om$ satisfying the condition of Lemma \ref{loca91}, and by duality one can deduce that 
$v$ is $C^{1,\alpha}$ in the relative interior of free boundary $\p V\cap\Lm$ with the constant independent of $d$. 
Therefore, the $C^{1,\alpha}$ norm of the free boundary $\partial V\cap \Lm$ is independent of $d$.
}
\end{remark}

\subsubsection{Closeness of potentials}
Recall that the active region $V$ satisfies \eqref{gq2} and $\delta\to0$ as $d\to\infty$, namely $V\to V_\infty$ as $d\to\infty$.
In order to compare $v$ and $v_\infty$, by adding a constant, we may assume
$v(x_0)=v_\infty(x_0)$ for some $x_0\in V\cap V_\infty$.
Let $V'\Subset V\cap V_\infty$, we have the following estimates:
\begin{lemma}\label{clolem}
There exists a positive, decreasing function $\omega$ satisfying $\omega(d)\rightarrow 0$, as $d\rightarrow \infty$, such that
\beq\label{C0est}
	\|v-v_\infty\|_{L^\infty(V')}\leq \omega(d).
\eeq
Moreover, if $y\in V'$ and $y+t^{1/2}e\in V'$, where $t>0$ is a constant and $e\in \mathbb{S}^{n-1}$, then when $d$ is sufficiently large such that $\omega(d)<t^{(1+\alpha)/2}$, we have
\beq\label{C1est}
	\left|\left(Dv(y)-Dv_\infty(y)\right)\cdot e\right| \leq C t^{\frac{\alpha}{2}},
\eeq
where the constant $C$ and the exponent $\alpha\in(0,1)$ are independent of $d$, as in Remark \ref{uniC1alpha}.
\end{lemma}

\begin{proof}
The estimate \eqref{C0est} follows from a standard compactness argument. We refer the reader to, for example, \cite[Lemma 4.1]{CF}.
It remains to prove \eqref{C1est}. 
From Lemma \ref{limsm}, the limit potential $v_\infty$ is $C^{1,\alpha}$ and thus
	$$ v_\infty(y+t^{\frac12}e) \leq v_\infty(y) + t^{\frac12} e\cdot Dv_\infty(y) + Ct^{\frac12(1+\alpha)}. $$
Meanwhile, by the convexity of $v$ we have
	$$ v(y+t^{\frac12}e) \geq v(y) + t^{\frac12} e\cdot Dv(y). $$
Then, from \eqref{C0est} we obtain
	\begin{equation*} 
	\begin{split}
		\left(Dv(y)-Dv_\infty(y)\right)\cdot e &\leq Ct^{\frac{\alpha}{2}} + 2t^{-\frac12}\omega(d) \\
			&\leq Ct^{\frac{\alpha}{2}}. 
	\end{split}
	\end{equation*}
By exchanging $v$ and $v_\infty$, we also have $\left(Dv(y)-Dv_\infty(y)\right)\cdot e \geq -Ct^{\frac{\alpha}{2}}$. 
\end{proof}

\subsection{Perturbation estimates}\label{s43}

Now we are ready to obtain higher regularity of $v$ up to the free boundary $\p V\cap\Lm$. 
The strategy is that: in \S\ref{locali} we first localise the problem by normalising a small sub-level set of $v_\infty$ at a point on $\p V_\infty\cap\Lm$; then in \S\ref{appro} we show that $v$ will be even closer to a convex function $w$ solving an optimal transport with constant densities, which enables an iteration argument to apply in \S\ref{s44}.  

\subsubsection{Localisation and normalisation}\label{locali}

Let $y_0\in\p V_\infty\cap\Lm$, and $x_0=Dv_\infty(y_0)\in \p U_\infty\cap\pom$. 
By translating the coordinates we may assume that $x_0=y_0=0$. 
Thanks to the obliqueness Lemma \ref{obliqueness} we can also assume that $\p U_\infty$ and $\p V_\infty$ have the unit inner normal $-e_n$ at the origin. 
Let $h>0$ small, and $S_{v_\infty, h}$ be the sub-level set of $v_\infty$ at the origin. 
By Lemma \ref{loc11}, up to an affine transformation $A$ with $\|A\|, \|A^{-1}\|\leq h^{-\epsilon}$, we also have
\begin{eqnarray}\label{da1}
 && C^{-1}B_{\sqrt{h}}\cap V_\infty\subset S_{v_\infty,h}\subset CB_{\sqrt{h}}\cap V_\infty, \\
 && C^{-1}B_{\sqrt{h}}\cap U_\infty\subset Dv_\infty(S_{v_\infty, h})\subset CB_{\sqrt{h}}\cap U_\infty. \nonumber
\end{eqnarray}

By Lemma \ref{clolem}, we can estimate the image of $S_{v_\infty,h}$ under the mapping $Dv$. 
In fact, set $t=h^{2/\alpha}$ in Lemma \ref{clolem}. 
For $y\in S_{v_\infty,h}$, if $y^n\leq-t^{1/2}$, then $y+t^{1/2}e\in V$ for any $e\in\mathbb{S}^{n-1}$, and thus from \eqref{C1est}
\beq\label{C1diff} 
	|Dv(y)-Dv_\infty(y)| \leq Ct^{\frac{\alpha}{2}} \leq Ch. 
\eeq
If $-t^{1/2}<y^n<0$, one can see that $y+t^{1/2}e\in V$, provided $e\in\mathbb{S}^{n-1}$ satisfies
\beq\label{gdic} 
	-e_n\cdot e > \theta(d)
\eeq
for some nonnegative, decreasing function $\theta\leq1$ with $\theta\to0$ as $d\to\infty$. 
Thus from \eqref{C1est}, $|(Dv(y)-Dv_\infty(y))\cdot e| \leq Ct^{\frac{\alpha}{2}} \leq Ch$ for those $e$ satisfying \eqref{gdic}. 
Meanwhile, by the $C^{1,\alpha}$ regularity of $v_\infty$, we have
$$ \dist(Dv_\infty(y), \p U_\infty) < Ct^{\frac{\alpha}{2}} \leq Ch. $$
Therefore, noting that $U\to U_\infty$ and $V\to V_\infty$, as $d\to\infty$, together with \eqref{C1diff} we obtain
\beq\label{shape1}
 Dv(S_{v_\infty,h}\cap V) \subset CB_{\sqrt{h}}\cap U.  
\eeq

On the other hand, by Remark \ref{reeq} and a similar argument as above applying to the dual potential $u$, we can also obtain 
\begin{equation}\label{shape2}
Du(C^{-1}B_{\sqrt{h}}\cap U)\subset S_{v_\infty,h}\cap V. 
\end{equation}

Combining \eqref{shape1} and \eqref{shape2}, we have 
\begin{equation}\label{shape3}
C^{-1}B_{\sqrt{h}}\cap U \subset Dv(S_{v_\infty,h}\cap V)\subset CB_{\sqrt{h}}\cap U.
\end{equation}
Since $V$ is arbitrarily close to $V_\infty$ provided $d$ is sufficiently large, from \eqref{da1} by enlarging $C$ slightly we also have
\begin{equation}\label{shape4}
C^{-1}B_{\sqrt{h}} \cap V\subset S_{v_\infty,h}\cap V\subset CB_{\sqrt{h}}\cap V.
\end{equation}

Having \eqref{shape3} and \eqref{shape4}, we make the rescaling 
\beq 
	x\mapsto \frac{x}{\sqrt{h}},\qquad  y\mapsto\frac{y}{\sqrt{h}}, \label{dilate1} 
\eeq
correspondingly
\beq 
	u(x) \to \frac{1}{h} u(\sqrt{h}x),\qquad  v(y) \to \frac{1}{h}v(\sqrt{h}y),  \label{dilate2} 
\eeq
and the densities become
$$ f(x) \to f(\sqrt{h}x),\qquad g(y) \to g(\sqrt{h}y). $$
Under the above rescaling, we have
\begin{eqnarray*}
&& Dv(S_{v_\infty,h}\cap V) \longrightarrow  \frac{1}{\sqrt{h}}Dv(S_{v_\infty,h}\cap V) =: \mathcal{C}_1, \\
&& S_{v_\infty,h}\cap V \longrightarrow \frac{1}{\sqrt{h}}(S_{v_\infty,h}\cap V) =:\mathcal{C}_2. 
\end{eqnarray*}

Hence, by \eqref{shape3} and \eqref{shape4} we have the initial setting of domains
\begin{equation}\label{shape5}
B_{1/C}\cap U\subset \mathcal{C}_1\subset B_C\cap U,
\end{equation}
and
\begin{equation}\label{shape006}
B_{1/C}\cap V\subset \mathcal{C}_2\subset B_C\cap V.
\end{equation}
 
\vskip 5pt

Next we show that how the above transformation $A$ and dilation \eqref{dilate1} would affect the fixed boundary $\p U\cap\pom$ and the free boundary $\p V\cap\Lm$ locally near the origin. 

Let $y_1:=\{te^n:t\in (-\delta, \delta)\}\cap \partial V,$ where $\delta\to0$ as $d\to\infty$.
Denote $x_1:=Dv(y_1)$.
By $C^{1,\alpha}$ continuity of $v$ we also have $x_1\to 0$ as $d\to\infty$. 
It is easy to see that the unit outer normal of $V$ (resp. $U$) at $y_1$ (resp. $x_1$) converges to $e_n$ as $d\to\infty$.
By a translation of coordinates: $$x\rightarrow x+Dv(y_1),\ y\rightarrow y+y_1,$$ and an affine transformation $A'$:
$$U\rightarrow A'U,\ V\rightarrow A'^{-1}V$$ 
we can assume that the normal of $V$ at $y_1$ is the same to that of $U$ at $x_1$. 
Note that the scale of translation converges to $0$, and $A'$ converges to the identity matrix as $d\rightarrow \infty$. 
These above transformations preserve the initial settings \eqref{shape3}, \eqref{shape4} (by enlarging $C$ slightly).
And under these changes, we may assume that $x_1=y_1=0$ and the unit outer normal of $U, V$ at $0$ is $e_n$

Hence, locally near the origin, $\p U$ and $\p V$ can be represented as 
\beq\label{locbdry}
\p U=\{\{x^n = P(x')\}\},\quad \p V=\{y^n = Q(y')\},
\eeq
for two functions $P\in C^{1,1}$ and $Q\in C^{1,\alpha}$ satisfying $P(0)=Q(0)=0$, $DP(0)=DQ(0)=0$.
Since the transformation $A$ satisfies $\|A\|, \|A^{-1}\|\leq Ch^{-\epsilon}$ for $\epsilon$ as small as we want and the dilation 
 \eqref{dilate1}, \eqref{dilate2} is of scale $h^{1/2}$, under these above changes, one has
\beq\label{locbdry2}
\|P\|_{C^{1,1}(\p U\cap\p{\mathcal{C}}_1)} \leq \delta_0, \quad \|Q\|_{C^{1,\alpha}(\p V\cap\p\mathcal{C}_2)} \leq \delta_0,
\eeq
where $\delta_0$ can be as small as we want, provided $h$ is sufficiently small.

\vskip5pt

The following lemma is a useful tool to derive the regularity of $u$. The proof follows from that of \cite[Theorem 2.1]{CF} by an iteration argument. For the sake of brevity, we omit it here and will give more details in \S4.4.1.
 \begin{lemma}\label{le111}
 Let $\mathcal{C}_1, \mathcal{C}_2, P, Q$ be as above.
Let $f, g$ be two densities supported in $\mathcal{C}_1$ and $\mathcal{C}_2$, respectively.
Let $v: \mathcal{C}_2\rightarrow \mathbb{R}$ be
a convex function such that $\partial v(\mathcal{C}_2)\subset B_C$ and $(Dv)_{\#}g=f$.
Then, for any $\beta\in (0,1)$, there exist constants $\delta_0$,
$\eta_0>0$ such that if: 
\begin{equation}\label{401}
\|P\|_{C^{1,1}(\p U\cap\p\mathcal{C}_1)}+ \|Q\|_{C^{1,\alpha}(\p V\cap\p\mathcal{C}_2)}\leq \delta_0,
\end{equation}
\begin{equation}\label{402}
\|f- {1}\|_{L^\infty(\mathcal C_1)}+\|g- {1}\|_{L^\infty(\mathcal C_2)}\leq \delta_0,
\end{equation}
and \begin{equation}\label{eq:uP}
\biggl\|v-\frac{1}{2}|y|^2\biggr\|_{L^\infty(\mathcal{C}_2)}\leq \eta_0,
\end{equation}
we can obtain $v\in C^{1,\beta}(\overline{\mathcal{C}_2\cap B_{r_0}})$, for a small constant $r_0>0$.
\end{lemma}

\begin{remark}\label{rrmk}
\emph{
Note that in \cite[Theorem 2.1]{CF}, it was assumed $\|P\|_{C^{2}}+ \|Q\|_{C^2}\leq \delta_0,$ but actually $C^{1,\alpha}$ bound is
sufficient. Indeed, the property needed in the iteration argument is that when we successively apply the affine transformations and dilations to normalise sub-level sets of height $h$, the boundaries $\p U\cap\p\mathcal{C}_1$ and $\p V\cap\p\mathcal{C}_2$ will become flatter and flatter as $h\to0$. This can be easily verified when \eqref{401} holds. 
}
\end{remark}

Therefore, in order to obtain the $C^{1,\beta}$ regularity of $v$, it suffices to verify the conditions \eqref{401}, \eqref{402}, and \eqref{eq:uP} after normalising a sub-level set $S_{v_\infty,h}$ for $h>0$ small, that enables the iteration argument in \S\ref{s44} to apply. 
We will verify these conditions in \S\ref{appro}. 

\subsubsection{An approximate problem}\label{appro}
Note that by the rescaling \eqref{dilate1} --\eqref{dilate2}, the original sub-level set $S_{v_\infty,h}$ becomes $S_{v_\infty,1}=\{y\in V_\infty : v_\infty(y)<1\}$.
Denote $S^*_1:=S_{v_\infty,1}$, and $B^-_r=B_r\cap\{x^n<0\}$.
Let 
$$\mathcal{D}_1:=\mathcal{C}_1\cup B^-_{1/C},\quad \mathcal{D}_2:= \rho S^*_1,$$
where $\rho>0$ is a constant chosen such that $|\mathcal{D}_1|=|\mathcal{D}_2|$.
Note that the dilation \eqref{dilate1} does not change the ``flat" free boundary $\p V_\infty\cap\Lm$, but makes the fixed boundary $\pom$ flatter and flatter near the origin. 
In fact,
\begin{equation*}
\mathcal{D}_1\backslash\mathcal{C}_1\subset B_{1/C}\cap \{-\delta<x^n<0\},
\end{equation*}
where $\delta \to 0$ as $h \to 0$, and moreover, the constant $\rho= {|\mathcal{D}_1|}/{|S^*_1|} \to 1$ as $h\to 0$, due to the measure preserving condition.  
When $d=\dist(\Omega, {\Omega^*})$ is sufficiently large, we have $|\mathcal{D}_1-\mathcal{C}_1|$ and $|\mathcal{D}_2-\mathcal{C}_2|$ is as small as we want. 
For densities, the dilation \eqref{dilate1} makes $f\chi_{\mathcal{C}_1}$ (resp. $g\chi_{\mathcal{C}_2}$) as close to $\chi_{\mathcal{C}_1}$ (resp. $\chi_{\mathcal{C}_2}$) as we want, provided $h$ is small enough. 

\vskip5pt
By the above observations, we construct an approximate optimal transport problem as follows.  
Let ${w}$ be the convex function solving 
$$(D{w})_\#\chi_{\mathcal{D}_1}=\chi_ {\mathcal{D}_2}$$
and ${w}(0)=u(0)$.
Then, by a standard compactness argument we have
\begin{lemma}\label{newC0}
$\|u-{w}\|_{L^\infty(\mathcal{C}_1)}\leq \delta_2,$
where $\delta_2$ can be as small as we want, provided $h$ is small enough and $d$ is large enough.
\end{lemma}

In order to show ${w}$ is indeed smooth near $0$, we use a symmetrisation method.
Let $$\tilde{\mathcal{D}}_1:=\mathcal{D}_1\cup (\mathcal{D}_1)^+$$
and $$\tilde{\mathcal{D}}_2:=\mathcal{D}_2\cup (\mathcal{D}_2)^+,$$ where $(E)^+$ denotes the reflection of the set $E$ with respect
to the hyperplane $\{x^n=0\}.$
Since $D_nv_\infty\leq 0$ in $\mathcal{D}_2$, we see that $\tilde{\mathcal{D}}_2$ is convex. 

Let $\tilde{w}$ be the convex function solving 
	$$(D\tilde{w})_\# \chi_{\tilde{\mathcal{D}}_1} = \chi_{\tilde{\mathcal{D}}_2}$$
and $\tilde{w}(0)={w}(0)$.
By symmetry and the uniqueness of optimal transport maps, we have $\tilde{w}={w}$ in $\mathcal{D}_1$. 
Since the target $\tilde{\mathcal{D}}_2$ is convex, by Caffarelli's interior regularity results we have $\tilde{w}\in C^3(B_{\frac{1}{2C}})$ and 
$\|\tilde{w}\|_{C^3(B_{\frac{1}{2C}})}\leq M$ for some universal constant $M$. 
Hence, as $\tilde w(0)=0$, $D\tilde w(0)=0$, locally near the origin
	$$\tilde{w}(x)=\frac{1}{2}D^2\tilde{w}(0)x \cdot x+O(|x|^3).$$ 
Note that by symmetry, we have $\tilde{w}_{n}=0$ along $\{x^n=0\}$, hence
$\tilde{w}_{\alpha n}(0)=0$ for $\alpha=1, \cdots,n-1.$ 
By an affine transformation preserving the $e_n$ direction, we can assume that
 \begin{equation}\label{taylor1}
 \tilde{w}(x)=\frac{1}{2}|x|^2+O(|x|^3).
 \end{equation}
 
The following lemma gives a local estimate.
\begin{lemma}\label{localC0}
For any $\eta>0$ small, there exists small positive constant $\epsilon_0$ such that
 \begin{eqnarray}
&&\|u-\frac{1}{2}|x|^2\|_{L^{\infty}(B_{\epsilon_0}\cap U)}\leq \eta \epsilon_0^2, \label{taylor2} \\
&&\|v-\frac{1}{2}|y|^2\|_{L^{\infty}(B_{\epsilon_0}\cap V)}\leq \eta \epsilon_0^2, \label{taylor3}
 \end{eqnarray}
provided $d$ is large enough. 
\end{lemma}

\begin{proof}
From \eqref{taylor1} one can see that
$$ \{\tilde w < h\} \approx B_{\sqrt{h}}\quad\mbox{ and }\quad D\tilde w\left(\{\tilde w<h\}\right)\approx B_{\sqrt{h}}. $$
Then by Lemma \ref{newC0} and the fact $\tilde w=w$ in $\mathcal{C}_1$, we have
$$ \|u-\frac{1}{2}|x|^2\|_{L^{\infty}(B_{\epsilon_0}\cap U)}\leq \delta_2 + C \epsilon_0^3, $$
where $\delta_2\to0$ as $d\to\infty$. 
Therefore, by taking $\epsilon_0$ sufficiently small and $d$ sufficiently large we can obtain \eqref{taylor2}.
Since $v$ is the Legendre transform of $u$, by duality we can also obtain \eqref{taylor3}.
\end{proof}

\subsection{Proof of Theorem \ref{thm2}}\label{s44}

\subsubsection{$C^{1,\beta}$ regularity}
For any given $\beta\in(0,1)$, we first prove that $v\in C^{1,\beta}$ near the origin. 
Let $\delta_0, \eta_0$ be the constants in conditions of Lemma \ref{le111}.
Thanks to Lemma \ref{localC0}, there exists a constant $\epsilon_0$ such that \eqref{taylor3} holds. 

Let $h_0=\epsilon_0^2$.
Note that up to an affine transformation $A$ with $\|A\|, \|A^{-1}\|\leq Ch_0^{-\epsilon}$, the sub-level set $S_{v_\infty, h_0}\approx B_{\epsilon_0}\cap V$. Hence, setting $h=h_0$ in \S4.3.1, and similarly to \eqref{dilate1}--\eqref{dilate2} making the rescaling
\beq\label{dilate11}
	v(y)\rightarrow \frac{1}{h_0}u(\epsilon_0y)
\eeq
and
\beq\label{dilate22}
	U\rightarrow \frac{1}{\epsilon_0}U,\ V\rightarrow \frac{1}{\epsilon_0}V,
\eeq
by Lemma \ref{localC0} and the argument as in \S4.3.1, we see that the conditions \eqref{401}--\eqref{eq:uP} are all satisfied. 
Therefore, by Lemma \ref{le111} we obtain that $v\in C^{1,\beta}$ near the origin. 

In fact, by the transformation $A$ and rescaling \eqref{dilate11}--\eqref{dilate22}, we again have the same setting as the initial setting \eqref{shape5} and \eqref{shape006}, so that we can apply the argument in \S4.3.2. Then by the iteration, noting that $\|A\|, \|A^{-1}\|\leq Ch_0^{-\epsilon}$ we can obtain
$$ B_{\left(\frac{\sqrt{h_0}}{Ch_0^{-\epsilon}}\right)^k}\cap V \subset S_{h_0^k}[v] \subset B_{\left(Ch_0^{-\epsilon}\sqrt{h_0}\right)^k}\cap V  $$
for any $k\geq1$. 
Let 
$$r_0:=h_0^{\frac12+\epsilon}/C, $$
then, since $\epsilon$ is as small as we want due to Lemma \ref{limsm}, we have for any given $\beta\in(0,1)$
$$ \|v\|_{L^\infty(B_{r_0^k}\cap V)} \leq h_0^k = (Ch_0^{-\epsilon}r_0)^{2k} \leq r_0^{(1+\beta)k}, $$
provided $h_0$ (and so $r_0$) is sufficiently small. 
This implies that $C^{1,\beta}$ regularity of $v$ near the origin. 

By rescaling back to the original solution and the arbitrariness of the chosen $y_0=0$  on the free boundary $\p V\cap\Lm$, we obtain that $v$ is $C^{1,\beta}$ up to the free boundary $\p V\cap\Lm$. 
Finally by recalling that the inner normal of the free boundary at $y\in\p V\cap\Lm$ is given by $Dv(y)-y$, we can conclude that the proof of Part $i)$ of Theorem \ref{thm2} by exchanging $u$ and $v$.
\qed

\subsubsection{$C^{2,\alpha}$ regularity}
When the densities $f, g$ are $C^\alpha$, we can use the argument of \cite[\S5]{CLW1} to obtain the $C^{2,\alpha}$ regularity of potentials up to the free boundaries and thus the $C^{2,\alpha}$ regularity of the free boundaries for the same exponent $\alpha\in(0,1)$.

The crucial ingredient is a finer local estimate in Lemma \ref{localC0}, that is for any $\epsilon_0>0$,
 \begin{eqnarray}
&&\|u-\frac{1}{2}|x|^2\|_{L^{\infty}(B_{\epsilon_0}\cap U)}\leq \epsilon_0^{2+\alpha}, \label{taylor222} \\
&&\|v-\frac{1}{2}|y|^2\|_{L^{\infty}(B_{\epsilon_0}\cap V)}\leq \epsilon_0^{2+\alpha}. \label{taylor333}
 \end{eqnarray}
 Let $h_0=\epsilon_0^2$. 
Since $v\in C^{1,\beta}$ near the origin, one has $S_{h_0,v}\approx B_{\epsilon_0}\cap V$ and moreover
$$ D_{h_0}^- := S_{h_0,v}\cap\{y^n\leq -h_0^{1-3\tau}\} \Subset V, $$
for any small $\tau>0$. 
Let $D_{h_0}^+$ be the reflection of $D_{h_0}^-$ with respect to the hyperplane $\{y^n=-h^{1-3\tau}\}$, and denote $D_{h_0}=D_{h_0}^+\cup D_{h_0}^-$. 
 In order to prove \eqref{taylor333}, we compare $v$ with the solution $w$ solving 
 \begin{equation}\label{approDiri}
 \left\{\begin{array}{ll}
 \det(D^2w)=1 \quad & \mbox{in } D_{h_0}, \\
 w=h_0\quad &\mbox{on }\p D_{h_0}.
 \end{array}
 \right.
 \end{equation}
By a similar argument as that of Theorem 1.1 (i) in \cite[\S5]{CLW1}, we can obtain that
$$ \|v-w\|_{L^\infty(D_{h_0}\cap U)} \leq Ch_0^{1+\frac{\alpha}{2}}. $$
By an affine transformation and rescaling back, we then have \eqref{taylor333}.

Since \eqref{taylor333} holds for any small $\epsilon_0$ (thus for any small $h_0$), the perturbation argument in \cite[\S5]{CLW1} applies. 
Here, we outline the main steps as follows. 
Let $D_k=D_{h_k}$, where $h_k=4^{-k}$, $k=0,1,2,\cdots$. 
Let $v_k$ be the convex solution of
 \begin{equation}\label{Diri1}
 \left\{\begin{array}{ll}
 \det(D^2v_k)=1 \quad & \mbox{in } D_k, \\
 v_k=h_k\quad &\mbox{on }\p D_k.
 \end{array}
 \right.
 \end{equation}
Then from \eqref{taylor333} and Schauder estimate (see  \cite[Lemma 5.4]{CLW1}), we have
$$ |D^2v_i(y)-D^2v_{i+1}(y)| \leq Ch_i^{\alpha/2},\quad y\in D_{i+2}. $$
Summing it over $k$, we have
\begin{equation}\label{C11}
\begin{split} 
|D^2v(0)| &\leq |D^2v_0(0)| + \sum_{i=0}^\infty |D^2v_i(0) - D^2v_{i+1}(0)| \\
	&\leq C+\sum_{i=0}^\infty Ch_i^{\alpha/2} \leq C_1,  
\end{split}
\end{equation}
for a universal constant $C_1$. This implies $v$ is $C^{1,1}$ up to the free boundary $\p V\cap\Lm$. 
Once having the estimate \eqref{C11}, heuristically the Monge-Amp\`ere equation 
$$ \det\, D^2v = \frac{g}{f(Dv)} $$
becomes uniformly elliptic. 
One can actually follow the argument in \cite[\S5]{CLW1} to obtain that if $f, g\in C^\alpha$, the solution $v$ is $C^{2,\alpha}$ up to the free boundary $\p V\cap\Lm$, which implies the $C^{2,\alpha}$ regularity of the free boundary $\p V\cap\Lm$ as well. 
Therefore, we conclude the proof of Part $ii)$ of Theorem \ref{thm2} by exchanging $u$ and $v$.
The proof of Part $iii)$ follows from the standard theory of elliptic equations \cite{GT}. 
\qed


\section{Application on another model}\label{S5}

We consider an optimal transport problem from the source domain $U$ associated with density $f$ to 
the target $V=V_1\cup V_2$ associated with density $g$, where $V_1, V_2$ are two domains separated by a hyperplane $\mathcal{H}$, and the densities satisfy 
$$ \int_U f(x)\,dx = \int_V g(y)\,dy $$
and $1/\lambda<f, g<\lambda$ for some positive constant $\lambda$.
Denote by $u$ (resp. $v$) the convex function solving $(\partial u)_\# f\chi_U=g\chi_V$ (resp. $(\partial v)_\# g\chi_V=f\chi_U$).

\begin{theorem}
The interior of $U_1:=\partial v(V_1)$ and $U_2:=\partial v(V_2)$ are disjoint and separated by a Lipschitz hypersurface.
\end{theorem}

\begin{proof}
Without loss of generality, we may assume that 
$$\mathcal{H}=\{y^n=0\},\quad V_1\subset \{y^n<0\}, \quad \mbox{and} \quad V_2\subset\{y^n>0\}.$$
Let
$$ \mathcal{D}:=\left\{\frac{y_2-y_1}{|y_2-y_1|} :  y_1\in V_1, y_2\in V_2\right\}. $$
From the assumption, there exists a small constant $\alpha>0$ such that
$$ \mathcal{D}\subset \left\{e\in\mathbb{S}^{n-1} : e\cdot e_n>\alpha \right\},$$
namely $\mathcal{D}$ is compactly included in the open upper hemisphere.
Define the cone
\beq\label{uniLip}
 	\mathcal{C}:=\left\{ e\in\mathbb{S}^{n-1} : e\cdot (-e_n)\geq \sqrt{1-\alpha^2} \right\}. 
\eeq
A straightforward computation shows that $z_1\cdot z_2<0$ for any $z_1\in \mathcal{D}$ and $z_2\in \mathcal{C}.$

Fix any $x\in U_1,$ by definition there exists some $y_1\in V_1$ such that $x\in\partial v(y_1).$ 
Hence, $y_1\in \partial u(x).$
 Denote 
 $\mathcal{C}_x:=\{x+z : z\in \mathcal{C}\}.$
Then, for any $\tilde{x}\in \mathcal{C}_x\cap U,$ we have 
 $$(\tilde{x}-x)\cdot (y_2-y_1)<0,\qquad \forall\, y_2\in V_2.$$ 
 On the other hand, by monotonicity of convex function, we have 
 $$(\tilde{x}-x)\cdot (\tilde y-y_1)\geq 0$$ for any $\tilde y\in \partial u(\tilde{x}).$
 Hence, 
 $\partial u(\mathcal{C}_x\cap U)\cap V_2=\emptyset,$ which implies that $\mathcal{C}_x\cap U\subset U_1.$ Therefore, we get a characterisation of $U_1$ as follows:
 $$U_1=\bigcup_{x\in U_1}\mathcal{C}_x\cap U.$$
 Denote by $f_x$ the Lipschitz function over $\{x^n=0\}$ with graph $\partial \mathcal{C}_x.$
 Let $$f:=\sup_{x\in U_1}f_x.$$ 
From \eqref{uniLip}, the function $f_x$ has a uniform Lipschitz bound, and thus $f$ is also a Lipschitz function.
Moreover, we have
 $$U_1=\left\{x^n\leq f(x^1,\cdots, x^{n-1})\right\}\cap U.$$
\end{proof}

\begin{theorem}\label{t2}
Suppose that $U$, $V_1$ and $V_2$ are {convex}, then the free boundary $\mathcal{F}:=\partial U_1\cap U$ is $C^1.$
\end{theorem}

\begin{proof}
Let $u$ be the potential function of the optimal transport problem from $U$ to $V$. 
Denote by $u_i$ the restriction of $u$ on $U_i$, $i=1, 2$. Note that $u_1=u_2$ on $\mathcal{F}$.
Similarly to Corollary \ref{coro1}, in order to prove $\mathcal{F}\in C^1$, it suffices to prove that $u_i$ is $C^1$ up to the free boundary $\mathcal{F}$ in the sense of Lemma \ref{C1lem}. 

The proof is similar to that of Lemma \ref{C1lem}. For completeness, we outline some key steps as follows. 
Suppose to the contrary that at some point $x_0\in \mathcal{F}$, $\partial \tilde{u}_i(x_0)$ contains two points $y_1, y_2\in V_i$. 
Without loss of generality, we may assume that $x_0$ is the origin. 
By duality, $0\in\p v_i(y_1)\cap\p v_i(y_2)$.
Since $v_i$ is strictly convex inside $V_i$, one has $y_1, y_2\in\p V_i$. 
Subtracting a constant to $v_i$, we may assume that $v_i\geq0=\min v_i$ and the segment 
$$ \overline{y_1 y_2} \subset \{v_i=0\} =: \mathcal{C}_0.$$
The contact set $\mathcal{C}_0$ satisfies three Claims in the proof of Lemma \ref{C1lem}.

Let $q_0\in\p V_i$ be an exposed point of $\mathcal{C}_0$ and $y_3:=\frac{1}{2}(y_1+y_2)$.
Denote $q_\delta:=(1-\delta)q_0+\delta y_3$. 
Consider the centred section $S_\varepsilon^c(q_\delta)$ of $v_i$ at $q_\delta$ with height $\varepsilon>0$ small. 
Heuristically, by a normalisation one has the observation that
$$ \frac{\dist(q_0, \p S_\varepsilon^c(q_\delta))}{\diam(S_\varepsilon^c(q_\delta))} \to 0,\quad \mbox{ as } \varepsilon\to0,$$
which implies $v_i(q_0)\to0$ as $\varepsilon\to0$. 
However, on the other hand, $v_i(q_0)$ is close to the minimum of $v_i$, thus $|v_i(q_0)|\approx 1$. 
This contradiction implies that $u_i$ is $C^1$ up to the free boundary $\mathcal{F}$. 

Since $V_1$ and $V_2$ are disjoint, we have $D\tilde{u}_1(x)\ne D\tilde{u}_2(x)$ for any $x\in \mathcal{F}$.
Therefore, by implicit function theorem we obtain that 
 the free boundary $\mathcal{F}$ is $C^1$.
\end{proof}

\begin{remark}
\emph{
Note that in the proof of Theorem \ref{t2}, we also see that the unit normal of $\mathcal{F}$ at $x$ is given by $\frac{Du_1(x)-Du_2(x)}{|Du_1(x)-Du_2(x)|}.$
Hence, higher regularity of $u_i$ will automatically imply higher regularity of $\mathcal{F}$.
}
\end{remark}

Let $v$ be the potential function of the optimal transport problem from $(V, g)$ to $(U, f).$ We extend $v$ to $\mathbb{R}^n$ as follows
$$v(x):=\sup\{L(x): L\ \text{is linear},\ L|_{V}\leq v,\ DL\in U\}.$$
The following localisation lemma is a key ingredient of obtaining the strict convexity of $v$, which in turn will implies the regularity of $u$.  
\begin{lemma}\label{2loc}
Suppose $U, V_1$ and $V_2$ are {convex}. 
Let $x\in\mathcal{F}\cap U$ and $y=Du_1(x)\in\p V_1$.
Let $R>0$ be a constant such that $B_R(y)$ contains no preimages $(Dv)^{-1}(\mathcal{F}\cap \partial U)$.
There exists a universal constant $h_0>0$ such that if $h<h_0$,
$$S^c_{v, h}(y)\subset B_R(y).$$ 
\end{lemma}

\begin{remark}
\emph{
Indeed, we can prove the above lemma in a larger set, namely, away from the preimages of tangential intersections of $\mathcal{F}$ and $\p U$. 
For simplicity, here we only state it for the preimages of the free boundary lying in the interior of $U$. 
Note that Lemma \ref{2loc} also holds replacing $V_1$ by $V_2.$
}
\end{remark}

\begin{proof} 
The proof is similar to that of Lemma \ref{1loc}, we include it here for reader's convenience.
Suppose to the contrary that there exists a sequence $h_k\rightarrow 0$ such that  $S^c_{v, h_k}(y)\not\subset B_R(y).$
By subtracting a linear function and translating the coordinates, we can assume $y=0$, $v\geq 0$ and $v(0)=0$. 
Then, $S^c_{v, h_k}(0)$ locally converges in Hausdorff distance to a convex set $Z$ balanced around the origin.
It is easy to see that $Z\subset \{v=0\}$, and $Z\not\subset B_R(0)$.
Hence, $\{v=0\}$ is a convex set containing more than one point, in particular, $\{v=0\}$ contains a segment balanced around $0$. 
 
Since the Monge-Amp\`ere measure $\det\,D^2v$ vanishes outside $V$, the set of extreme points $ext(\{v=0\})$ must be contained in $\overline{V}.$ On the other hand, since $v$ is strictly convex in the interior of $V$, we have $ext(\{v=0\})\subset \partial V$. 
Note that $0\notin ext(\{v=0\})$.
Since $V_1$ and $V_2$ are separated by a hyperplane, there must exists an extreme point $\hat y\in ext(\{v=0\})\cap\p V_1$.  
However, this contradicts with the $C^1$ regularity of $u_1$ in Theorem \ref{t2}.
\end{proof}

Once having the above localisation lemma, we can adopt the argument as in \cite{C92,CM} to obtain boundary $C^{1,\alpha}$ regularity. 
Similarly to \S\ref{S3}, we can establish the interior $C^{1,\alpha}$ regularity of the free boundary $\mathcal{F}$.

\vskip10pt
Now, observe that when $d:=\dist(V_1, V_2)$ is sufficiently large, $\frac{Du_2(x)-Du_1(x)}{|Du_2(x)-Du_1(x)|}$ is uniformly close to the unit vector $e_n$, for all $x \in \mathcal{F}$.
Hence, the free boundary is close to a hyperplane (as close as we want, provided $d$ is large enough). Then, we can follow our argument in \S\ref{S4} to establish the following theorem.

\begin{theorem}\label{maint1}
Let $U, V_1, V_2, \mathcal{F}$ be as above. Assume that $\p U, \p V_i \in C^{1,1}$ are convex.  

\noindent a) When $f, g\in C^0$, for any $\beta\in(0,1)$, there exists a constant $d_\beta>0$ such that if $d>d_\beta$, the free boundary
 $\mathcal{F}$ is $C^{1,\beta}$.  

\noindent b) When $f, g\in C^{\alpha}$ for some $\alpha\in(0,1)$, there exists a constant $d_\alpha>0$ such that if $d>d_\alpha$,
the free boundary $\mathcal{F}$ is $C^{2,\alpha}$.

\noindent c) When $U, V, f, g$ are smooth, the free boundary $\mathcal{F}$ is $C^\infty$ in the interior of $U$, provided $d$ is sufficiently large. 
\end{theorem}

\bibliographystyle{amsplain}

\end{document}